\documentclass[12pt, a4paper]{article}
\usepackage{amssymb}

\newtheorem{theorem}{Theorem}[section]
\newtheorem{lemma}[theorem]{Lemma}
\newtheorem{corollary}[theorem]{Corollary}

\title{{\Large \bf  The domination number and the least $Q$-eigenvalue\thanks{Supported by NSFC
(Nos. 11271315, 11171290, 11226290, 11201417).}}}
\author{Guanglong Yu$^a$\thanks{E-mail addresses:
yglong01@163.com, yctusgguo@gmail.com, zhangrongzcx@126.com.}
~ Shu-Guang Guo$^{a}$~ Rong Zhang$^{a}$ ~ Yarong Wu$^{b}$\thanks{Corresponding author: wuyarong1@163.com.} ~
\\ ~ \\
{\footnotesize $^a$Department of Mathematics, Yancheng Teachers
University,}\\ {\footnotesize  Yancheng, 224002, Jiangsu, China}\\
{\footnotesize $^b$SMU college of art and science, Shanghai maritime
University, Shanghai, 200135, China}}
\date{}

\begin{document}
%\openup 1.0\jot
\maketitle

\begin{abstract}
A vertex set $D$ of a graph $G$ is said to be a dominating set if every vertex of $V(G)\setminus D$ is
adjacent to at least a vertex in $D$, and the domination number $\gamma(G)$ ($\gamma$, for short) is the
minimum cardinality of all dominating sets of $G$. For a graph, the least $Q$-eigenvalue is the least
eigenvalue of its signless Laplacian matrix. In this paper, for a nonbipartite graph with both
order $n$ and domination number $\gamma$, we show that
$n\geq 3\gamma-1$, and show that it contains a unicyclic spanning subgraph
with the same domination number $\gamma$.
By investigating the relation between the domination number and the least $Q$-eigenvalue of a graph,
we minimize the least $Q$-eigenvalue among all the nonbipartite graphs with given domination number.


\end{abstract}
\baselineskip 18.6pt

\section{Introduction}

\ \ \ \ All graphs considered in this paper are connected, undirected and
simple, i.e. no loops or multiple edges are allowed. We denote by $|S|$ the cardinality of a set $S$,
and denote by $G=G[V(G)$, $E(G)]$ a graph with vertex set
$V(G)$ and edge set $E(G)$, where $|V(G)|= n$ is the order
and $|E(G)|= m$ is the size.
Recall
that $Q(G)= D(G) + A(G)$ is called the $signless$ $Laplacian$ $matrix$ of $G$, where $D(G)= \mathrm{diag}(d_{1}, d_{2},
\ldots, d_{n})$ with $d_{i}= d_{G}(v_{i})$ being the degree of
vertex $v_{i}$ $(1\leq i\leq n)$, and $A(G)$ is the adjacency matrix of $G$. The least eigenvalue of $Q(G)$,
denote by $q_{min}(G)$, is called the $least$ $Q$-$eigenvalue$ of $G$. Noting that $Q(G)$ is positive semi-definite,
we have $q_{min}(G)\geq 0.$

The signless Laplacian matrix has received a lot of attention in recent years, especially after D. Cvetkovi\'{c} et al.
put forward the study of
this matrix in [2-8]. From \cite{D.P.S}, we know that, for a connected
graph $G$, $q_{min}(G)= 0$ if and only if $G$ is bipartite.
Consequently, in \cite{DR}, the least $Q$-eigenvalue was studied as a measure of nonbipartiteness of a graph. One can note
that there
are quite a few results about the least $Q$-eigenvalue. In \cite{CCRS}, Domingos M. Cardoso et al. determined the the graphs
with the the minimum least $Q$-eigenvalue among all the connected nonbipartite
graphs with a prescribed number of vertices. In \cite{LOA}, L. de Lima et al. surveyed some known results about $q_{min}$ and
also proved some
new ones; at the end they stated some open problems. In \cite{FF}, S. Fallat, Y. Fan investigated the relations
between the least $Q$-eigenvalue and some parameters reflecting
the graph bipartiteness. In \cite{WF}, Y. Wang, Y. Fan investigated the least $Q$-eigenvalue
of a graph under some perturbations, and minimized the least eigenvalue of the signless
Laplacian among the class of connected graphs with fixed order
which contains a given nonbipartite graph as an induced subgraph.

Recall that if a vertex $u$ is adjacent to a vertex $v$ in a graph, we say that $u$ $dominates$ $v$ or $v$ $dominates$ $u$.
A vertex set $D$ of a graph $G$ is said to be a $dominating$ $set$ if every vertex of $V(G)\setminus D$ is
adjacent to (dominated by) at least a vertex in $D$, and the $domination$ $number$ $\gamma(G)$ ($\gamma$, for short) is the
minimum cardinality of all dominating sets of $G$. In a graph $G$, we say that a vertex $v$ is $dominated$ by a vertex set
$S$ if $v\in S$ or $v$ is adjacent to a vertex in $S$. A graph $H$ is said to be $dominated$ by a vertex set $S$ if every
vertex of $H$ is dominated by $S$. Clearly, a graph is dominated by its any dominating set.

A connected graph $G$ of order $n$ is called a $unicyclic$ graph if
$|E(G)|=n$. A $unicyclic$ $spanning$ $subgraph$ of
a graph is its a spanning subgraph which is unicyclic.
It is known that for a connected graph $G$ of order $n$, $\gamma\leq \frac{n}{2}$ (see \cite{ORE}).
In this paper, for a nonbipartite graph with both order $n$ and domination number $\gamma$, we show that
$n\geq 3\gamma-1$, and show that it contains a unicyclic spanning subgraph
with the same domination number $\gamma$.

Denote by $C_{k}$ a $k$-cycle (of length $k$). If $k$ is odd, we say $C_{k}$ an $odd$ $cycle$.
For an odd number $s\ge 3$, we let $C_{s,\, l}^*$ be the graph of order $n$ obtained by attaching a
cycle $C_s$ to an end vertex
of a path $P_{l+1}$ and attaching $n-s-l$ pendant edges to the other end vertex of the path $P_{l+1}$ (see Fig. 1.1).
In particular, $l=0$ means attaching $n-s$ pendant edges to a vertex of $C_s$.

\setlength{\unitlength}{0.5pt}
\begin{center}
\begin{picture}(612,149)
\qbezier(17,84)(17,104)(42,119)\qbezier(42,119)(67,133)(104,133)\qbezier(104,133)(140,133)(165,119)
\qbezier(165,119)(191,104)(191,84)\qbezier(191,84)(191,63)(165,48)\qbezier(165,48)(140,34)(104,34)
\qbezier(104,35)(67,35)(42,48)\qbezier(42,48)(16,63)(17,84)
\put(191,81){\circle*{4}}
\put(356,81){\circle*{4}}
\put(376,81){\circle*{4}}
\put(394,81){\circle*{4}}
\put(265,81){\circle*{4}}
\put(577,89){\circle*{4}}
\put(333,81){\circle*{4}}
\qbezier(191,81)(262,81)(333,81)
\put(413,82){\circle*{4}}
\put(493,82){\circle*{4}}
\qbezier(413,82)(453,82)(493,82)
\put(577,74){\circle*{4}}
\put(577,59){\circle*{4}}
\put(577,135){\circle*{4}}
\qbezier(493,82)(535,109)(577,135)
\put(577,107){\circle*{4}}
\qbezier(493,82)(535,95)(577,107)
\put(576,40){\circle*{4}}
\qbezier(493,82)(534,61)(576,40)
\put(180,109){\circle*{4}}
\put(174,55){\circle*{4}}
\put(21,71){\circle*{4}}
\put(21,99){\circle*{4}}
\put(42,50){\circle*{4}}
\put(42,118){\circle*{4}}
\put(195,88){$v_{1}$}
\put(245,91){$v_{s+1}$}
\put(311,91){$v_{s+2}$}
\put(392,92){$v_{s+l-1}$}
\put(465,94){$v_{s+l}$}
\put(179,116){$v_{2}$}
\put(176,42){$v_{s}$}
\put(8,30){$v_{\lceil\frac{s}{2}\rceil+2}$}
\put(-42,65){$v_{\lceil\frac{s}{2}\rceil+1}$}
\put(-22,105){$v_{\lceil\frac{s}{2}\rceil}$}
\put(15,134){$v_{\lfloor\frac{s}{2}\rfloor}$}
\put(582,36){$v_{n}$}
\put(583,140){$v_{s+l+1}$}
\put(583,110){$v_{s+l+2}$}
\put(260,-9){Fig. 1.1. $C_{s,\, l}^{\ast}$}
\end{picture}
\end{center}

By
investigating the relation between the structure of a graph
and the domination number, and investigating how the least $Q$-eigenvalue of a graph
changes under some perturbations, we consider the
relation between the least $Q$-eigenvalue and the domination number, showing that
among all the nonbipartite graphs with both order $n$ and domination number $\gamma$,
(i) if $n=3\gamma-1$, $3\gamma$, $3\gamma+1$, then the graph with the minimal least $Q$-eigenvalue attains uniquely
at  $C_{3,\, n-4}^*$ (see Fig. 1.2);
(ii) if $n\geq3\gamma+2$, then the graph with the minimal least $Q$-eigenvalue attains uniquely at
$C_{3,\, 3\gamma-3}^*$ (see Fig. 1.2).

\setlength{\unitlength}{0.5pt}
\begin{center}
\begin{picture}(643,153)
\put(8,131){\circle*{4}}
\put(8,57){\circle*{4}}
\qbezier(8,131)(8,94)(8,57)
\put(47,93){\circle*{4}}
\qbezier(8,131)(27,112)(47,93)
\qbezier(8,57)(27,75)(47,93)
\put(85,93){\circle*{4}}
\qbezier(47,93)(66,93)(85,93)
\put(120,93){\circle*{4}}
\qbezier(85,93)(102,93)(120,93)
\put(132,92){\circle*{4}}
\put(142,92){\circle*{4}}
\put(152,92){\circle*{4}}
\put(389,90){\circle*{4}}
\put(444,89){\circle*{4}}
\put(163,92){\circle*{4}}
\put(205,92){\circle*{4}}
\qbezier(163,92)(184,92)(205,92)
\put(47,102){$v_{1}$}
\put(1,140){$v_{2}$}
\put(4,40){$v_{3}$}
\put(293,124){\circle*{4}}
\put(293,55){\circle*{4}}
\qbezier(293,124)(293,90)(293,55)
\put(346,90){\circle*{4}}
\qbezier(293,124)(319,107)(346,90)
\qbezier(293,55)(319,73)(346,90)
\put(427,90){\circle*{4}}
\qbezier(346,90)(386,90)(427,90)
\put(459,89){\circle*{4}}
\put(472,89){\circle*{4}}
\put(486,90){\circle*{4}}
\put(531,90){\circle*{4}}
\qbezier(486,90)(508,90)(531,90)
\put(587,130){\circle*{4}}
\qbezier(531,90)(559,110)(587,130)
\put(597,110){\circle*{4}}
\qbezier(531,90)(564,100)(597,110)
\put(588,53){\circle*{4}}
\qbezier(531,90)(559,72)(588,53)
\put(594,71){\circle*{4}}
\put(596,82){\circle*{4}}
\put(597,93){\circle*{4}}
\put(80,103){$v_{4}$}
\put(156,106){$v_{n-1}$}
\put(211,91){$v_{n}$}
\put(345,99){$v_{1}$}
\put(287,133){$v_{2}$}
\put(289,38){$v_{3}$}
\put(383,99){$v_{4}$}
\put(510,102){$v_{3\gamma}$}
\put(588,140){$v_{3\gamma+1}$}
\put(604,110){$v_{3\gamma+2}$}
\put(595,47){$v_{n}$}
\put(95,23){$C^{\ast}_{3,n-4}$}
\put(422,23){$C^{\ast}_{3,3\gamma-3}$}
\put(160,-20){Fig. 1.2. two extremal graphs}
\end{picture}
\end{center}

\section{Preliminary}

\ \ \ \ In this section, we introduce some notations and some working lemmas.

We denote by $P_n$ a $path$ of order $n$,
 $K_{r,s}$ the $complete$ $bipartite$ graph with partite
sets of order $r$ and $s$. For a path $P$ and a cycle
$C$, we denote by $l(P)$, $l(C)$ their $lengths$ respectively.
The $distance$ between two vertices
$u$ and $v$ in a graph $G$, denoted by $d_{G}(u, v)$, is the length of the shortest path from $u$ to $v$; the
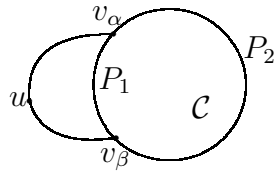
$distance$ between two subgraphs $G_{1}$ and $G_{2}$, denoted by $d_{G}(G_{1}, G_{1})$, is the length of the
shortest path from $G_{1}$ to $G_{2}$. Clearly, $d_{G}(G_{1}, G_{1})=\min\{d_{G}(u, v)\, |\, u\in V(G_{1}), v\in V(G_{2})\}$.
The $girth$ of a graph $G$, denoted by $g(G)$, is the length of the shortest cycle in $G$. For a nonbipartite graph $G$,
the $odd$ $girth$, denoted by $g_{o}(G)$, is the length of the shortest odd cycle.  Let $G-uv$
 denote the graph that arises from $G$ by deleting the edge $uv\in
 E(G)$, and let $G-v$
 denote the graph that arises from $G$ by deleting the vertex $v\in V(G)$ and the edges incident with $v$.
 Similarly,  $G+uv$ is the graph that arises from $G$ by adding an edge $uv$ between its two nonadjacent vertices $u$
 and $v$. For an edge set $E$, we let $G-E$ denote the graph obtained by deleting all the edges in $E$ from $G$. A $pendant$ $vertex$
 is a vertex of degree $1$. A vertex is called a $pendant$ $neighbor$ if it is adjacent to a pendant vertex.
 The $union$ of two simple graphs $H$ and $G$ is the simple graph $G\cup H$ with vertex set $V(G)\cup V(H)$
 and edge set $E(G)\cup E(H)$.

Let $G_1$ and $G_2$ be two disjoint graphs, and let $v_1\in V(G_1),$  $v_2\in V(G_2)$. The $coalescence$ of $G_1$ and $G_2$,
denoted by $G_1(v_1)\diamond G_2(v_2)$, is obtained from $G_1$, $G_2$ by identifying $v_1$ with $v_2$ and forming a new vertex
$u$ (see \cite{WF} for detail). The graph $G_1(v_1)\diamond G_2(v_2)$ is also written as $G_1(u)\diamond G_2(u)$. If a connected
graph $G$ can be expressed
in the form $G = G_1(u)\diamond G_2(u)$, where $G_1$ and $G_2$ are both nontrivial and connected, then for $i=1$, $2$, $G_i$ is
called a
$branch$ of $G$ with root $u$.

Let $G$ be a graph of order $n$, $X=(x_1, x_2, \ldots, x_n)^T \in R^n$ be defined on $V(G)$, that is,
each vertex $v_i$ is mapped to $x_i =x(v_i)$; let $|x(v_i)|$ denote the absolute value of $x(v_i)$.
One can find
that $X^TQ(G)X =\sum_{uv\in E(G)}[x(u) + x(v)]^2.$
In addition, for an arbitrary unit vector $X\in R^n$, $q_{min}(G) \le X^TQ(G)X$,
with equality if and only if $X$ is an eigenvector corresponding to $q_{min}(G)$.
For convenience, an eigenvector of $Q(G)$ sometimes be called an eigenvector of $G$.
A branch $H$ of $G$ is called a $zero$ $branch$ with respect to $X$ if $x(v) = 0$  for all $v \in V(H)$; otherwise, it
is called a $nonzero$ $branch$ with respect to $X$.

\begin{lemma}{\bf (\cite{D.R.S})} \label{le2,0} %------
Let $G$ be a graph on $n$ vertices and $m$ edges, and let $e$ be an
edge of $G$. Let $q_{1}\geq q_{2}\geq \cdots \geq q_{n}$ and
$s_{1}\geq s_{2}\geq \cdots \geq s_{n}$ be the $Q$-eigenvalues of  $G$
and $G-e$ respectively. Then $0\leq s_{n}\leq q_{n}\leq \cdots \leq
s_{2}\leq q_{2}\leq s_{1}\leq q_{1}.$
\end{lemma}

\begin{lemma}{\bf (\cite{WF})}\label{le2,1} %--
Let $G$ be a connected graph which contains a bipartite branch $H$ with root $u$.  Let $X$ be an eigenvector of $G$ corresponding
to $\kappa(G)$.

{\normalfont (i)} If $x(u) = 0$, then $H$ is a zero branch of G with respect to $X$;

{\normalfont (ii)} If $x(u)\neq 0$, then $x(p)\neq 0$ for every vertex $p\in V(H)$. Furthermore, for every vertex $p\in V(H)$,
$x(p)x(u)$ is either positive or negative, depending on whether $p$ is or is not in the same part of the bipartite graph $H$ as
$u$; consequently, $x(p)x(q) < 0$ for each edge $pq \in E(H)$.
\end{lemma}

\begin{lemma}{\bf (\cite{WF})}\label{le2,4} %------
Let $G$ be a connected nonbipartite graph of order $n$, and let $X$ be an eigenvector of $G$ corresponding to $\kappa(G)$.
Let $T$ be a
tree, which is a nonzero branch of $G$ with respect to $X$ and with root $u$. Then $|x(q)| < |x(p)|$
 whenever $p,$ $q$
are vertices of $T$ such that $q$ lies on the unique path from $u$ to $p$.
\end{lemma}

\begin{lemma}{\bf (\cite{YGX})}\label{le2,3} %------
 Let $G = G_1(v_2) \diamond T(u)$ and $G^* = G_1(v_1)\diamond T(u)$, where $G_1$ is a connected nonbipartite
graph containing two distinct vertices $v_1, v_2$, and $T$ is a nontrivial tree. If there exists an
 eigenvector $X=(\,x(v_1)$, $x(v_2)$, $\ldots$, $x(v_k)$, $\ldots)^T$ of $G$ corresponding to $\kappa(G)$ such that
 $|x(v_1)| > |x(v_2)|$ or $|x(v_1)| = |x(v_2)| > 0$, then $\kappa(G^*)<\kappa(G)$.
\end{lemma}

\setlength{\unitlength}{0.6pt}
\begin{center}
\begin{picture}(577,166)
\put(48,78){\circle*{4}}
\put(86,124){\circle*{4}}
\qbezier(48,78)(67,101)(86,124)
\put(120,78){\circle*{4}}
\qbezier(86,124)(104,101)(122,78)
\qbezier(48,78)(85,78)(122,78)
\qbezier(0,77)(0,85)(6,91)\qbezier(6,91)(13,97)(24,97)\qbezier(24,97)(34,97)(41,91)\qbezier(41,91)(48,85)(48,77)
\qbezier(48,77)(48,68)(41,62)\qbezier(41,62)(34,57)(24,57)\qbezier(24,57)(13,57)(6,62)\qbezier(6,62)(0,68)(0,77)
\qbezier(119,82)(119,91)(126,97)\qbezier(126,97)(134,104)(145,104)\qbezier(145,104)(155,104)(163,97)
\qbezier(163,97)(171,91)(171,82)\qbezier(171,82)(171,72)(163,66)\qbezier(163,66)(155,60)(145,60)
\qbezier(145,60)(134,60)(126,66)\qbezier(126,66)(118,72)(119,82)
\qbezier(62,142)(62,149)(69,155)\qbezier(69,155)(77,161)(88,161)\qbezier(88,161)(98,161)(106,155)
\qbezier(106,155)(114,149)(114,142)\qbezier(114,142)(114,134)(106,128)\qbezier(106,128)(98,123)(88,123)
\qbezier(88,123)(77,123)(69,128)\qbezier(69,128)(61,134)(62,142)
\put(79,106){$v_{1}$}
\put(49,66){$v_{2}$}
\put(102,66){$v_{3}$}
\put(80,141){$T_{1}$}
\put(14,76){$T_{2}$}
\put(139,81){$T_{3}$}
\put(30,10){$\mathcal {C}^{(T_{1}, T_{2}, T_{3}; 1,2,3)}_{3}$}
\put(274,81){\circle*{4}}
\put(311,128){\circle*{4}}
\qbezier(274,81)(292,105)(311,128)
\put(346,90){\circle*{4}}
\qbezier(311,128)(328,109)(346,90)
\qbezier(274,81)(310,86)(346,90)
\put(230,122){\circle*{4}}
\qbezier(274,81)(278,145)(230,122)
\qbezier(274,81)(208,87)(230,122)
\put(238,42){\circle*{4}}
\qbezier(274,81)(219,82)(238,42)
\qbezier(274,81)(292,36)(238,42)
\qbezier(288,147)(288,155)(295,160)\qbezier(295,160)(303,166)(314,166)\qbezier(314,166)(324,166)(332,160)
\qbezier(332,160)(340,155)(340,147)\qbezier(340,147)(340,138)(332,133)\qbezier(332,133)(324,127)(314,127)
\qbezier(314,128)(303,128)(295,133)\qbezier(295,133)(287,138)(288,147)
\put(303,109){$v_{1}$}
\put(281,68){$v_{2}$}
\put(305,146){$T_{1}$}
\put(241,106){$T_{2}$}
\put(244,55){$T_{3}$}
\put(354,90){$v_{3}$}
\put(238,10){$\mathcal {C}^{(T_{1}, T_{2}, T_{3}; 1,2,2)}_{3}$}
\put(431,59){\circle*{4}}
\put(486,102){\circle*{4}}
\qbezier(431,59)(458,81)(486,102)
\put(535,59){\circle*{4}}
\qbezier(486,102)(510,81)(535,59)
\qbezier(431,59)(483,59)(535,59)
\put(427,120){\circle*{4}}
\qbezier(486,102)(413,70)(427,120)
\qbezier(427,120)(444,154)(486,102)
\put(492,157){\circle*{4}}
\qbezier(486,102)(446,157)(492,157)
\qbezier(486,102)(531,151)(492,157)
\put(546,103){\circle*{4}}
\qbezier(486,102)(540,149)(546,103)
\qbezier(486,102)(547,62)(546,103)
\put(477,86){$v_{1}$}
\put(437,109){$T_{1}$}
\put(481,138){$T_{2}$}
\put(522,103){$T_{3}$}
\put(417,50){$v_{2}$}
\put(540,56){$v_{3}$}
\put(444,10){$\mathcal {C}^{(T_{1}, T_{2}, T_{3}; 1)}_{3}$}
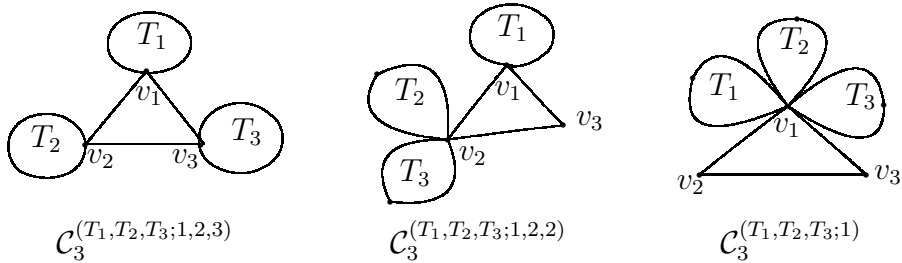
\put(80,-27){Fig. 2.1. $\mathcal {C}^{(T_{1}, T_{2}, T_{3}; 1,2,3)}_{3}$, $\mathcal {C}^{(T_{1}, T_{2}, T_{3}; 1,2,2)}_{3}$,
$\mathcal {C}^{(T_{1}, T_{2}, T_{3}; 1)}_{3}$}
\end{picture}
\end{center}

\vspace{0.35cm}

Let $k\geq 3$ be odd, and let $\mathcal {C}=v_1v_2\cdots v_kv_1$ be a cycle of length $k$. For $j=1$, $2$, $\ldots$, $t$,
each $T_{j}$ is a nontrivial tree. Let $\mathcal {C}^{(T_{1}, T_{2}, \ldots, T_{t};i_{1}, i_{2}, \ldots, i_{t})}_{k}$
denote the graph obtained by identifying the vertex $u_{j}$ of $T_{j}$ and the vertex $v_{i_{j}}$ of $\mathcal {C}$,
where $1\leq j\leq t$ and for $1\leq l< f\leq t$, $i_{l}=i_{f}$ possibly. Here, in
$\mathcal {C}^{(T_{1}, T_{2}, \ldots, T_{t};i_{1}, i_{2}, \ldots, i_{t})}_{k}$, for any $1\leq j\leq t$,
we denote by $v_{i_{j}}$ the new vertex obtained by identifying $v_{i_{j}}$ and $u_{j}$. Let $\mathcal {C}^{T_{1}, T_{2},
\ldots, T_{t}}_{(k, n)}=\{\mathcal {C}^{(T_{1}, T_{2}, \ldots, T_{t};i_{1}, i_{2}, \ldots, i_{t})}_{k}|\,
\mathcal {C}^{(T_{1}, T_{2}, \ldots, T_{t};i_{1}, i_{2}, \ldots, i_{t})}_{k}$ be of order $n\}$ and let
$\mathcal {C}^{(T_{1}, T_{2}, \ldots, T_{t};1)}_{k}=\mathcal {C}^{(T_{1}, T_{2}, \ldots, T_{t};1, 1, \ldots, 1)}_{k}$.
For understanding easily, we show three examples in Fig. 2.1.

\begin{lemma}\label{le2,7} %------
Let $k< n$ be odd and $\mathcal {C}^{(T_{1}, T_{2}, \ldots, T_{t};i_{1}, i_{2}, \ldots, i_{t})}_{k}$ be of order $n$.
$X=(\,x(v_1)$, $x(v_2)$, $\ldots$, $x(v_k)$, $x(v_{k+1})$, $x(v_{k+2})$, $\ldots$, $x(v_{n-1})$, $x(v_{n})\,)^T$ is a
unit eigenvector corresponding to $\kappa(\mathcal {C}^{(T_{1}, T_{2}, \ldots, T_{t};i_{1}, i_{2}, \ldots, i_{t})}_{k})$.
Then $\max\{|x(v_{i_{j}})|\, |\, 1\leq j\leq t\}>0$.
\end{lemma}

\begin{proof}
Assume that $\max\{|x(v_{i_{j}})|\, |\, 1\leq j\leq t\}=0$. By Lemma \ref{le2,1}, we know that $T_{1}, T_{2}, \ldots, T_{t}$
are all zero branch. Because $X\neq 0$, there must be a vertex $v_{s}\in V(\mathcal {C})$ such that
$x(v_{s})\neq 0$. Suppose that $i_{1}<s$ and $|i_{1}-s|=\min\{|i_{j}-s|\, |\, 1\leq j\leq t\}$. Let
$U=\mathcal {C}^{(T_{1}, T_{2}, \ldots, T_{t};s, a_{2}, a_{3}, \ldots, a_{t})}_{k}$, where for
$2\leq j\leq t$, $a_{j}\equiv i_{j}+|i_{1}-s|\ (\mathrm{mod}\ k)$, and denote by $v_{a_{j}}$ the new vertex
obtained by identifying $v_{a_{j}}$ and $u_{j}$ of $T_{j}$. Let $Y=(\,y(v_1)$, $y(v_2)$, $\ldots$, $y(v_k)$,
$y(v_{k+1})$, $y(v_{k+2})$, $\ldots$, $y(v_{n-1})$, $y(v_{n})\,)^T$ be a vector defined on $V(U)$ satisfying
$$y(w)=\left \{\begin{array}{ll}
 (-1)^{d_{T_{j}}(v_{a_{j}}, w)}x(v_{a_{j}}),\ & \ w\in V(T_{j})\ for\ j=1,\ 2,\ \ldots,\ t;
\\ x(w),\ & \ others. \end{array}\right.$$
Note that
$Y^{T}Q(U)Y=X^{T}Q(\mathcal {C}^{(T_{1}, T_{2}, \ldots, T_{t};i_{1}, i_{2}, \ldots, i_{t})}_{k})X$ and $Y^{T}Y>X^{T}X$. Then
$$\displaystyle \kappa(U)\leq\frac{Y^{T}Q(U)Y}{Y^{T}Y}<\frac{X^{T}Q(\mathcal {C}^{(T_{1}, T_{2},
\ldots, T_{t};i_{1}, i_{2}, \ldots, i_{t})}_{k})X}{X^{T}X}=\kappa(\mathcal {C}^{(T_{1}, T_{2}, \ldots, T_{t};
i_{1}, i_{2}, \ldots, i_{t})}_{k}).$$ This is a contradiction because $U\cong \mathcal {C}^{(T_{1}, T_{2},
\ldots, T_{t};i_{1}, i_{2}, \ldots, i_{t})}_{k}$. Then the result follows. \ \ \ \ \ $\square\square$

\end{proof}

\begin{lemma}{\bf (\cite{YGX})}\label{le2,11} %------
Let $3\leq s\leq n-2$ be odd, and let both $C_{s,\, l}^{\ast}$ and $C_{s,\, l+1}^{\ast}$ be of order $n$. Then
$\kappa(C_{s,\, l+1}^*)< \kappa(C_{s,\, l}^*).$
\end{lemma}

\setlength{\unitlength}{0.5pt}
\begin{center}
\begin{picture}(835,137)
\qbezier(22,80)(22,96)(39,107)\qbezier(39,107)(57,119)(82,119)\qbezier(82,119)(106,119)(124,107)
\qbezier(124,107)(142,96)(142,80)\qbezier(142,80)(142,63)(124,52)\qbezier(124,52)(106,40)(82,40)
\qbezier(82,41)(57,41)(39,52)\qbezier(39,52)(22,63)(22,80)
\put(141,79){\circle*{4}}
\put(197,79){\circle*{4}}
\qbezier(141,79)(169,79)(197,79)
\put(133,61){\circle*{4}}
\put(734,111){\circle*{4}}
\put(167,80){\circle*{4}}
\put(23,85){\circle*{4}}
\put(26,67){\circle*{4}}
\put(119,50){\circle*{4}}
\put(123,108){\circle*{4}}
\put(135,95){\circle*{4}}
\put(228,79){\circle*{4}}
\put(219,79){\circle*{4}}
\put(209,79){\circle*{4}}
\put(39,54){\circle*{4}}
\put(32,101){\circle*{4}}
\put(239,79){\circle*{4}}
\put(265,125){\circle*{4}}
\qbezier(239,79)(252,102)(265,125)
\put(290,79){\circle*{4}}
\qbezier(239,79)(264,79)(290,79)
\put(282,99){\circle*{4}}
\put(277,108){\circle*{4}}
\put(271,117){\circle*{4}}
\put(119,77){$v_{1}$}
\put(133,103){$v_{2}$}
\put(112,115){$v_{3}$}
\put(28,81){$v_{k+1}$}
\put(10,105){$v_{k}$}
\put(-19,64){$v_{k+2}$}
\put(0,42){$v_{k+3}$}
\put(120,36){$v_{2k}$}
\put(139,53){$v_{2k+1}$}
\put(149,86){$v_{2k+2}$}
\put(182,65){$v_{2k+3}$}
\put(286,90){\circle*{4}}
\qbezier(239,79)(262,85)(286,90)
\put(222,89){$v_{a}$}
\put(290,86){$v_{n-1}$}
\put(261,130){$v_{a+1}$}
\put(296,68){$v_{n}$}
\put(609,83){\circle*{4}}
\put(543,82){\circle*{4}}
\qbezier(337,87)(371,87)(405,87)
\qbezier(337,74)(370,74)(404,74)
\qbezier(390,96)(405,88)(421,79)
\qbezier(421,79)(405,73)(389,67)
\put(561,82){\circle*{4}}
\put(552,82){\circle*{4}}
\put(443,108){\circle*{4}}
\put(443,55){\circle*{4}}
\qbezier(443,108)(443,82)(443,55)
\put(490,82){\circle*{4}}
\qbezier(443,108)(466,95)(490,82)
\qbezier(443,55)(466,69)(490,82)
\put(533,82){\circle*{4}}
\qbezier(490,82)(511,82)(533,82)
\put(514,92){$v_{k-1}$}
\put(487,90){$v_{k}$}
\put(430,117){$v_{k+1}$}
\put(434,38){$v_{k+2}$}
\put(571,83){\circle*{4}}
\put(650,83){\circle*{4}}
\qbezier(571,83)(610,83)(650,83)
\put(563,91){$v_{2}$}
\put(601,91){$v_{1}$}
\put(627,91){$v_{2k+2}$}
\put(681,83){\circle*{4}}
\put(671,83){\circle*{4}}
\put(661,83){\circle*{4}}
\put(693,84){\circle*{4}}
\put(744,84){\circle*{4}}
\qbezier(693,84)(718,84)(744,84)
\put(720,135){\circle*{4}}
\qbezier(693,84)(706,110)(720,135)
\put(740,101){\circle*{4}}
\qbezier(693,84)(716,93)(740,101)
\put(729,119){\circle*{4}}
\put(725,126){\circle*{4}}
\put(681,68){$v_{a}$}
\put(715,142){$v_{a+1}$}
\put(745,101){$v_{n-1}$}
\put(750,80){$v_{n}$}
\put(750,63){\circle*{4}}
\qbezier(693,84)(721,74)(750,63)
\put(745,37){\circle*{4}}
\put(747,45){\circle*{4}}
\put(749,53){\circle*{4}}
\put(741,28){\circle*{4}}
\qbezier(693,84)(717,56)(741,28)
\put(756,57){$v_{k+3}$}
\put(743,14){$v_{2k+1}$}
\put(159,11){$C_{2k+1,\, l}^{\ast}$}
\put(573,13){$C_{3,\, t}^*$}
\put(300,-25){Fig. 2.2. $C_{2k+1,\, l}^{\ast},\ C_{3,\, t}^*$}
\end{picture}
\end{center}

\begin{lemma}{\bf (\cite{YGX})}\label{le2,13} %------
For $k\ge 1$, let both $C_{2k+1,\, l}^{\ast}$ and $C_{3,\, t}^*$ ($t=l+k-1$) be of order $n$. Then
$\kappa(C_{3,\, t}^*)\leq \kappa(C_{2k+1,\, l}^*)$, with equality if and only if $k=1$ (see Fig. 2.2).
\end{lemma}

\section{Domination number and the structure of a graph}

\begin{theorem}\label{th3,1} %------
Let $G$ be a nonbipartite graph with domination number $\gamma(G)$. Then $G$ contains a unicyclic spanning subgraph $H$
with both $g(H)=g_{o}(G)$ and $\gamma(G)=\gamma(H)$.
\end{theorem}

\begin{proof}
Denote by $D$ a dominating set of $G$ with cardinality $\gamma(G)$. We denote by $\mathcal {C}$ a cycle with length
$g_{o}(G)$ in $G$ which contains the largest number of vertices of $D$ among all cycles with length $g_{o}(G)$.
It is easy to see that there is no chord in $\mathcal {C}$ because otherwise, there is an odd cycle with length less
than $g_{o}(G)$.

Let $U=\{M\,|\, M$ be a subset of $D$ with the minimum cardinality among all the subsets of $D$ which
dominates $\mathcal {C}\}$, and let $D_{\mathcal {C}}\in U$ be a subset contains the maximum number of $\mathcal {C}$
among all the subsets in $U$. Let $S=D_{\mathcal {C}}\backslash V(\mathcal {C})$,
and let $F^{'}_{1}=G[V(\mathcal {C})\cup D_{\mathcal {C}}]-E(G[S])$.

{\bf Claim 1} If $S\neq \emptyset$,
then in $F^{'}_{1}$, every vertex in $S$ is pendent vertex. Otherwise,
suppose that there exists a vertex $u\in S$ which is adjacent to at least two vertices of $\mathcal {C}$, and suppose
$\mathcal {C}=v_{1}v_{2}\cdots v_{z}v_{1}$ (where $z=g_{o}(G)$). If $u$ is adjacent to two adjacent vertices of
$\mathcal {C}$,
say $v_{1}$, $v_{2}$ for convenience, then $g_{o}(G)=3$, and in $\mathcal {C}$, $z=3$. Now, we say that
$v_{3}$ is in
$D$ because otherwise, $\mathcal {C}^{'}=v_{1}v_{2}uv_{1}$ is a cycle containing more vertices of $D$
than $\mathcal {C}$, which contradicts the choice of $\mathcal {C}$. Note that both $v_{1}$ and $v_{2}$ are
dominated by $v_{3}$. $v_{3}\in D$ means that $|D_{\mathcal {C}}|=1$.
Noting the choice of $D_{\mathcal {C}}$, we get that $D_{\mathcal {C}}$ contains only one vertex
of $\mathcal {C}$ which is in $D$ ($D_{\mathcal {C}}=\{v_{3}\}$ possibly). Then $S= \emptyset$, which contradicts our
assumption that $S\neq \emptyset$. As a result, we get that if $S\neq \emptyset$,
then no vertex in $S$ is adjacent to two adjacent vertices of $\mathcal {C}$. This tells us that if $S\neq \emptyset$
and there exists a vertex in $S$ which is adjacent to at least two vertices of $\mathcal {C}$,
then the length of
$\mathcal {C}$ is at least $5$, that is, $g_{o}(G)\geq5$.

{\bf Assertion 1} If there exists a vertex in $S$ which is adjacent to two nonadjacent vertices of
$\mathcal {C}$, then one of the two paths
obtained by parting $\mathcal {C}$ with the two nonadjacent vertices is with length $2$.
To prove this assertion, we suppose that a vertex $u\in S$ is adjacent to
two vertices of $\mathcal {C}$, say $v_{\alpha}$, $v_{\beta}$.
Then $\mathcal {C}$ is partitioned into two path $P_{1}$ and $P_{2}$ by $v_{\alpha}$ and $v_{\beta}$,
 that is, $\mathcal {C}=P_{1}\cup P_{2}$, where $v_{\alpha}$, $v_{\beta}$ are the end vertices
of both $P_{1}$ and $P_{2}$ (see Fig. 3.1).

Assume this assertion can not hold. Then both $l(P_{1})$ and $l(P_{2})$ are more than $2$.
Note that one of $l(P_{1})$, $l(P_{2})$ is odd. Suppose $l(P_{1})$ is odd for convenience. Let $P^{'}=v_{\alpha}uv_{\beta}$.
Then $\mathcal {C}^{'}=P_{1}\cup P^{'}$ is an odd cycle with length less than $\mathcal {C}$,
which contradicts the choice that $l(\mathcal {C})=g_{o}(G)$. Assertion 1 is proved.

\setlength{\unitlength}{0.5pt}
\begin{center}
\begin{picture}(531,175)
\put(397,89){\circle*{4}}
\put(437,154){\circle*{4}}
\put(408,56){\circle*{4}}
\put(406,123){\circle*{4}}
\put(432,34){\circle*{4}}
\put(333,83){\circle*{4}}
\qbezier(333,83)(365,86)(397,89)
\put(340,-15){Fig. 3.2. $\mathcal {C}$ and $\xi$}
\put(490,69){$\mathcal {C}$}
\put(401,88){$v_{\alpha}$}
\put(412,119){$v_{0}$}
\put(430,165){$v_{\beta}$}
\put(413,55){$u_{0}$}
\put(423,18){$v_{\sigma}$}
\put(315,80){$\xi$}
\put(530,118){$\mathcal {P}^{''}$}
\qbezier(397,92)(397,64)(416,44)\qbezier(416,44)(436,25)(464,25)\qbezier(464,25)(491,25)(511,44)
\qbezier(511,44)(531,64)(531,92)\qbezier(397,92)(397,119)(416,139)\qbezier(416,139)(436,159)(464,159)
\qbezier(464,159)(491,159)(511,139)\qbezier(511,139)(531,119)(531,92)
\qbezier(333,83)(334,158)(437,154)
\qbezier(333,83)(336,23)(432,34)
\qbezier(63,88)(63,64)(79,47)\qbezier(79,47)(96,31)(120,31)\qbezier(120,31)(143,31)(160,47)\qbezier(160,47)(177,64)(177,88)
\qbezier(63,88)(63,111)(79,128)\qbezier(79,128)(96,145)(120,145)\qbezier(120,145)(143,145)(160,128)
\qbezier(160,128)(177,111)(177,88)
\put(78,126){\circle*{4}}
\put(80,48){\circle*{4}}
\put(15,76){\circle*{4}}
\qbezier(15,76)(11,127)(78,126)
\qbezier(15,76)(25,40)(80,48)
\put(62,137){$v_{\alpha}$}
\put(68,30){$v_{\beta}$}
\put(0,72){$u$}
\put(67,85){$P_{1}$}
\put(175,108){$P_{2}$}
\put(137,63){$\mathcal {C}$}
\put(2,-15){Fig. 3.1. $P_{1}$ and $P_{2}$}
\end{picture}
\end{center}

{\bf Assertion 2} No vertex in $S$ is adjacent to more than $2$ vertices of $\mathcal {C}$. Otherwise, assume that
there exists a vertex $\xi\in S$ which is adjacent to $3$ vertices of $\mathcal {C}$, say $v_{\alpha}$, $v_{\beta}$, $v_{\sigma}$
for convenience. Suppose $\mathcal {C}$ is parted into two paths $P_{1}$, $P_{2}$ by
$v_{\alpha}$, $v_{\beta}$. By Assertion 1, we know that one of
$l(P_{1})$, $l(P_{2})$ is $2$. For convenience, we assume that $l(P_{1})=2$, and
assume that $P_{1}=v_{\alpha}v_{0}v_{\beta}$ (see Fig. 3.2). Noting that no vertex in $S$ is
adjacent to two adjacent vertices of $\mathcal {C}$, we see that
$v_{\sigma}\in V(P_{2})$. Suppose that $\mathcal {C}$ is parted into two paths $\mathcal {P}_{1}$,
$\mathcal {P}_{2}$ by $v_{\alpha}$, $v_{\sigma}$. Then one of $\mathcal {P}_{1}$,
$\mathcal {P}_{2}$ contains $P_{1}$. For convenience, we assume that $\mathcal {P}_{1}$
contains $P_{1}$. Then $l(\mathcal {P}_{1})> 2$ (in fact, $l(\mathcal {P}_{1})\geq 4$).

By Assertion 1, we know
that $l(\mathcal {P}_{2})= 2$. Assume that
$\mathcal {P}_{2}=v_{\alpha}u_{0}v_{\sigma}$, and assume that $\mathcal {C}$ is parted into two paths
$\mathcal {P}^{'}$,
$\mathcal {P}^{''}$ by $v_{\beta}$, $v_{\sigma}$, where
$\mathcal {P}^{'}=P_{1}\cup \mathcal {P}_{2}$ (see Fig. 3.2). Because $l(\mathcal {P}^{'})=4$,
$l(\mathcal {P}^{''})$ is odd. Let $P=v_{\beta}\xi v_{\sigma}$. Then $P\cup \mathcal {P}^{''}$ is an odd cycle
with length less than $\mathcal {C}$, which contradicts the choice that $l(\mathcal {C})=g_{o}(G)$. Assertion 2 is proved.

The above two assertions tell us that a vertex in $S$ is adjacent to at most two vertices of
$\mathcal {C}$. Suppose that the vertex $\eta\in S$ is adjacent to two nonadjacent vertices $\omega_{1}$ and $\omega_{2}$
of $\mathcal {C}$. Then $\mathcal {C}$ is parted into two paths by $\omega_{1}$ and $\omega_{2}$.
Denote by $\mathbb{P}_{1}$ and $\mathbb{P}_{2}$ the two paths.
By Assertion 1, we know that one of $\mathbb{P}_{1}$, $\mathbb{P}_{2}$ is with length 2. Suppose that $l(\mathbb{P}_{1})=2$,
and suppose $\mathbb{P}_{1}=\omega_{1}\omega_{0}\omega_{2}$. Let $\mathbb{P}^{'}=\omega_{1}\eta\omega_{2}$. We say that
$\omega_{0}\in D$. Otherwise, $\mathbb{P}^{'}\cup \mathbb{P}_{2}$ is a cycle of length $g_{o}(G)$ which contains more
vertices of $D$ than $\mathcal {C}$, which contradicts the choice of $\mathcal {C}$. Let
$D^{'}_{\mathcal {C}}=D_{\mathcal {C}}\backslash \{\eta\}$. We find that $D^{'}_{\mathcal {C}}$ also dominates $\mathcal {C}$,
but this contradicts the choice of $D_{\mathcal {C}}$. This means that no vertex in $S$ is adjacent to two vertices of
$\mathcal {C}$.

From above all, we conclude that no vertex in $S$ is adjacent to more than one vertex of
$\mathcal {C}$. It means that every vertex in $S$ is pendent vertex in $F^{'}_{1}$. Claim 1 is proved.

Let $F=G[D\backslash V(F^{'}_{1})]$, and $F_{2}$, $\ldots$, $F_{k}$ be all the connected components of $F$.
By deleting edges, we can get $F^{'}_{2}$, $\ldots$, $F^{'}_{k}$ from $F_{2}$, $\ldots$, $F_{k}$ respectively,
such that $F^{'}_{2}$, $\ldots$, $F^{'}_{k}$ are all the trees. Suppose
$d_{G}(F^{'}_{1}$, $F^{'}_{2})=\min\{d_{G}(F^{'}_{1}$, $F^{'}_{i})\, |\, 2\leq i\leq k\}$.
Denote by $P_{F^{'}_{1}, F^{'}_{2}}$  one of the shortest paths from $F^{'}_{1}$ to $F^{'}_{2}$ in $G$, that is,
$P_{F^{'}_{1}, F^{'}_{2}}$ satisfies $l(P_{F^{'}_{1}, F^{'}_{2}})=d_{G}(F^{'}_{1}$, $F^{'}_{2})$.
Noting the assumption of $d_{G}(F^{'}_{1}$, $F^{'}_{2})$ and the choice of $P_{F^{'}_{1}, F^{'}_{2}}$,
we get that except the two end vertices of $P_{F^{'}_{1}, F^{'}_{2}}$,
all other vertices of
$P_{F^{'}_{1}, F^{'}_{2}}$ are in $V(G)\setminus V(F)$. Let $P_{F^{'}_{1}, F^{'}_{2}}=v_{i_{1}}v_{i_{2}}\cdots v_{i_{t}}$,
where $v_{i_{1}}\in V(F^{'}_{1})$,
$v_{i_{t}}\in V(F^{'}_{2})$.

\setlength{\unitlength}{0.5pt}
\begin{center}
\begin{picture}(437,206)
\qbezier(0,76)(0,98)(21,114)\qbezier(21,114)(43,130)(74,130)\qbezier(74,130)(104,130)(126,114)
\qbezier(126,114)(148,98)(148,76)\qbezier(148,76)(148,53)(126,37)\qbezier(126,37)(104,21)(74,21)
\qbezier(74,22)(43,22)(21,37)\qbezier(21,37)(0,53)(0,76)
\qbezier(283,70)(283,90)(305,105)\qbezier(305,105)(328,119)(360,119)\qbezier(360,119)(391,119)(414,105)
\qbezier(414,105)(437,90)(437,70)\qbezier(437,70)(437,49)(414,34)\qbezier(414,34)(391,20)(360,20)
\qbezier(360,21)(328,21)(305,34)\qbezier(305,34)(283,49)(283,70)
\qbezier(208,169)(208,184)(224,195)\qbezier(224,195)(241,206)(266,206)\qbezier(266,206)(290,206)(307,195)
\qbezier(307,195)(324,184)(324,169)\qbezier(324,169)(324,153)(307,142)\qbezier(307,142)(290,132)(266,132)
\qbezier(266,132)(241,132)(224,142)\qbezier(224,142)(207,153)(208,169)
\put(304,75){\circle*{4}}
\put(232,75){\circle*{4}}
\put(249,75){\circle*{4}}
\put(116,76){\circle*{4}}
\put(156,76){\circle*{4}}
\qbezier(116,76)(136,76)(156,76)
\put(190,76){\circle*{4}}
\qbezier(156,76)(173,76)(190,76)
\put(266,75){\circle*{4}}
\qbezier(304,75)(285,75)(266,75)
\put(214,75){\circle*{4}}
\put(256,152){\circle*{4}}
\qbezier(190,76)(223,114)(256,152)
\put(45,70){$F^{'}_{1}$}
\put(366,66){$F^{'}_{2}$}
\put(263,179){$F^{'}_{a}$}
\put(102,87){$v_{i_{1}}$}
\put(150,60){$v_{i_{2}}$}
\put(186,60){$v_{i_{3}}$}
\put(310,76){$v_{i_{t}}$}
\put(260,157){$v_{i_{b}}$}
\put(152,-9){Fig. 3.3. $P_{F^{'}_{1}, F^{'}_{2}}$}
\end{picture}
\end{center}

{\bf Claim 2} $l(P_{F^{'}_{1}, F^{'}_{2}})\leq 3$, that is $t\leq 4$. Otherwise, suppose that $l(P_{F^{'}_{1}, F^{'}_{2}})\geq 4$,
that is, $t\geq5$ (see Fig. 3.3). Note that $v_{i_{3}}$
must be adjacent to at least a vertex in $D$. Suppose $v_{i_{3}}$ is adjacent to a vertex $v_{i_{b}}$ of $F^{'}_{a}$
where $a\neq 1$, $2$. Then $d_{G}(F^{'}_{1},
F^{'}_{a})<d_{G}(F^{'}_{1},
F^{'}_{2})$. This contradicts that $d_{G}(F^{'}_{1}$,
$F^{'}_{2})=\min\{d_{G}(F^{'}_{1}$, $F^{'}_{i})\, |\, 2\leq i\leq k\}$. Hence, Claim 2 holds.

{\bf Case 1} $v_{i_{2}}$ is dominated by $V(F^{'}_{1}\cup F^{'}_{2})\cap D$. We assume
that $v_{i_{2}}$ is dominated by $v_{i_{0}}\in (V(F^{'}_{1}\cup F^{'}_{2})\cap D)$. If $v_{i_{0}}\in V(F^{'}_{1})$ and
$v_{i_{0}}\neq v_{i_{1}}$, we let
$P^{'}_{F^{'}_{1}, F^{'}_{2}}=v_{i_{0}}v_{i_{2}}\cdots v_{i_{t}}$, and let
$\mathcal {F}_{1}=F^{'}_{1}\cup P^{'}_{F^{'}_{1}, F^{'}_{2}}\cup F^{'}_{2}$; if $v_{i_{0}}= v_{i_{1}}$, we let
$\mathcal {F}_{1}=F^{'}_{1}\cup P_{F^{'}_{1}, F^{'}_{2}}\cup F^{'}_{2}$. If $v_{i_{0}}\in V(F^{'}_{2})$,
then $v_{i_{3}}=v_{i_{0}}$, $P_{F^{'}_{1}, F^{'}_{2}}=v_{i_{1}}v_{i_{2}}v_{i_{3}}$,
and then we let $\mathcal {F}_{1}=F^{'}_{1}\cup P_{F^{'}_{1}, F^{'}_{2}}\cup F^{'}_{2}$.

{\bf Case 2} $v_{i_{2}}$ is not dominated by $V(F^{'}_{1}\cup F^{'}_{2})\cap D$. Then $v_{i_{2}}$ must
be dominated by $D\backslash (V(F^{'}_{1}\cup F^{'}_{2})\cap D)$. That is, $v_{i_{2}}$ is dominated by
$\displaystyle V(\cup_{i=3}^{k}F^{'}_{i})$. Then $v_{i_{2}}\in \displaystyle V(\cup_{i=3}^{k}F^{'}_{i})$
or $v_{i_{2}}$ is adjacent to a vertex in $\displaystyle V(\cup_{i=3}^{k}F^{'}_{i})$.
If $v_{i_{2}}$ is adjacent to a vertex in $\displaystyle V(\cup_{i=3}^{k}F^{'}_{i})$, for convenience,
we assume that $v_{i_{2}}$ is dominated by a vertex $v_{i_{0}}$ in
$V(F^{'}_{3})$. Then we let
$\mathcal {F}_{1}=F^{'}_{1}\cup P_{F^{'}_{1}, F^{'}_{2}}\cup F^{'}_{2}\cup v_{i_{2}}v_{i_{0}}\cup F^{'}_{3}$.
If $v_{i_{2}}\in \displaystyle V(\cup_{i=3}^{k}F^{'}_{i})$, for convenience,
we assume that $v_{i_{2}}\in V(F^{'}_{3})$.
Then we let $\mathcal {F}_{1}=F^{'}_{1}\cup P_{F^{'}_{1}, F^{'}_{2}}\cup F^{'}_{2}\cup F^{'}_{3}$.

Now, $\mathcal {F}_{1}$ is a unicyclic graph and $D\cap V(\mathcal {F}_{1})$ is a
dominating set of $\mathcal {F}_{1}$.

For Case 1, suppose
$d_{G}(\mathcal {F}_{1}$, $F^{'}_{3})=\min\{d_{G}(\mathcal {F}_{1}$, $F^{'}_{i})\, |\, 3\leq i\leq k\}$. Denote by
$P_{\mathcal {F}_{1}, F^{'}_{3}}$ one of the shortest paths from $\mathcal {F}_{1}$ to $F^{'}_{3}$ in $G$. Similar to
the $P_{F^{'}_{1}, F^{'}_{2}}$, we can prove that $l(P_{\mathcal {F}_{1}, F^{'}_{3}})\leq 3$. Suppose
$P_{\mathcal {F}_{1}, F^{'}_{3}}=v_{j_{1}}v_{j_{2}}\cdots v_{j_{s}}$ where $s\leq 4$.

{\bf Subcase 1} $v_{j_{2}}$ is dominated by $V(\cup^{3}_{i=1} F^{'}_{3})\cap D$. For convenience,
we assume that $v_{j_{2}}$ is dominated by a vertex $v_{j_{0}}\in V(F^{'}_{1})$. If $v_{j_{0}}\neq v_{j_{1}}$,
we let $P^{'}_{\mathcal {F}_{1}, F^{'}_{3}}=v_{j_{0}}v_{j_{2}}\cdots v_{j_{t}}$, and we let
$\mathcal {F}_{2}=\mathcal {F}_{1}\cup P^{'}_{\mathcal {F}_{1}, F^{'}_{3}}\cup F^{'}_{3}$;
 if $v_{j_{0}}= v_{j_{1}}$, we let $\mathcal {F}_{2}=\mathcal {F}_{1}\cup P_{\mathcal {F}_{1}, F^{'}_{3}}\cup F^{'}_{3}$.

{\bf Subcase 2} $v_{j_{2}}$ is not dominated by $V(\cup^{3}_{i=1} F^{'}_{3})\cap D$. Then $v_{j_{2}}$ must
be dominated by $D\backslash (V(\cup^{3}_{i=1} F^{'}_{3})\cap D)$. That is, $v_{j_{2}}$ is dominated by
$\displaystyle V(\cup_{i=4}^{k}F^{'}_{i})$. For convenience, we assume that $v_{j_{2}}$ is dominated by a vertex
$v_{j_{0}}\in V(F^{'}_{4})$. Then we let
$\mathcal {F}_{2}=\mathcal {F}_{1}\cup P_{\mathcal {F}_{1}, F^{'}_{3}}\cup F^{'}_{3}\cup v_{j_{2}}v_{j_{0}}\cup F^{'}_{4}$.

Now, $\mathcal {F}_{2}$ is a unicyclic graph and $D\cap V(\mathcal {F}_{2})$ is its a
dominating set, where $|V(\mathcal {F}_{2})|>|V(\mathcal {F}_{1})|$. Similarly, for Case 2, we can get a unicyclic
graph $\mathcal {F}_{2}$ such that $D\cap V(\mathcal {F}_{2})$ is its a
dominating set and
$|V(\mathcal {F}_{2})|>|V(\mathcal {F}_{1})|$.

Proceeding like this, we can get a unicyclic graph $\mathcal {F}_{z}$ such that
$\displaystyle V(\cup_{i=1}^{k}F^{'}_{i})\subseteq V(\mathcal {F}_{z})$, and
 $D$ is also a dominating set of
$\mathcal {F}_{z}$.

Assume that $V(G)\setminus V(\mathcal {F}_{z})\neq \emptyset$ and assume that $V(G)\setminus V(\mathcal {F}_{z})=\{v_{a_{1}}$,
$v_{a_{2}}$, $\ldots$, $v_{a_{f}}\}$. Note that each vertex in $V(G)\setminus V(\mathcal {F}_{z})$ is adjacent to at least
one vertex in $D$. For $1\leq i\leq f$ and for each vertex $v_{a_{i}}$, we select only one vertex in $D$ which is
adjacent to $v_{a_{i}}$ in $G$, denote by $v_{b_{i}}$ (where $v_{b_{i}}s$ are not necessarily distinct). Let
$\displaystyle H=\mathcal {F}_{z}\cup(\cup^{f}_{i=1} v_{a_{i}}v_{b_{i}})$, where for $1\leq i\leq f$, $v_{a_{i}}v_{b_{i}}$
is a pendant edge. Then $H$ is a unicyclic spanning subgraph of $G$ with $g(H)=g_{o}(G)$, and $D$ is also a dominating
set of $H$. As a result, $\gamma(H)\leq\gamma(G)$. Noting that $H$ is a spanning subgraph of $G$, and any dominating set of $H$
is also a dominating set of $G$, we get that $\gamma(H)\geq\gamma(G)$. Then $\gamma(G)=\gamma(H)$ follows. This completes
the proof. \ \ \ \ \ $\Box$
\end{proof}

\setlength{\unitlength}{0.5pt}
\begin{center}
\begin{picture}(535,183)
\put(21,87){\circle*{4}}
\put(67,116){\circle*{4}}
\qbezier(21,87)(44,102)(67,116)
\put(131,116){\circle*{4}}
\qbezier(67,116)(99,116)(131,116)
\put(131,61){\circle*{4}}
\qbezier(131,116)(131,89)(131,61)
\put(63,61){\circle*{4}}
\qbezier(21,87)(42,74)(63,61)
\qbezier(63,61)(97,61)(131,61)
\put(78,161){\circle*{4}}
\qbezier(21,87)(49,124)(78,161)
\qbezier(78,161)(104,139)(131,116)
\put(185,61){\circle*{4}}
\qbezier(131,61)(158,61)(185,61)
\put(185,134){\circle*{4}}
\qbezier(185,61)(185,98)(185,134)
\qbezier(78,161)(131,148)(185,134)
\put(152,161){\circle*{4}}
\qbezier(78,161)(115,161)(152,161)
\put(71,173){$v_{1}$}
\put(-2,86){$v_{2}$}
\put(57,45){$v_{3}$}
\put(125,45){$v_{4}$}
\put(135,112){$v_{5}$}
\put(63,101){$v_{6}$}
\put(192,58){$v_{7}$}
\put(192,133){$v_{8}$}
\put(158,163){$v_{9}$}
\put(93,15){$G$}
\qbezier(238,97)(269,97)(300,97)
\qbezier(238,84)(268,84)(299,84)
\put(359,88){\circle*{4}}
\put(408,116){\circle*{4}}
\qbezier(359,88)(383,102)(408,116)
\put(404,62){\circle*{4}}
\qbezier(359,88)(381,75)(404,62)
\put(471,116){\circle*{4}}
\qbezier(408,116)(439,116)(471,116)
\put(471,62){\circle*{4}}
\qbezier(404,62)(437,62)(471,62)
\qbezier(471,116)(471,89)(471,62)
\put(530,62){\circle*{4}}
\qbezier(471,62)(500,62)(530,62)
\put(407,161){\circle*{4}}
\qbezier(359,88)(383,125)(407,161)
\put(478,161){\circle*{4}}
\qbezier(407,161)(442,161)(478,161)
\put(527,132){\circle*{4}}
\qbezier(407,161)(467,147)(527,132)
\put(401,174){$v_{1}$}
\put(334,86){$v_{2}$}
\put(397,45){$v_{3}$}
\put(463,45){$v_{4}$}
\put(476,112){$v_{5}$}
\put(404,101){$v_{6}$}
\put(536,58){$v_{7}$}
\put(532,131){$v_{8}$}
\put(483,163){$v_{9}$}
\put(451,15){$H$}
\put(185,-17){Fig. 3.4. $G$ and $H$}
\qbezier(282,106)(299,99)(317,92)
\qbezier(317,92)(299,85)(282,78)
\end{picture}
\end{center}

{\bf Remark}
For a nonbipartite graph $G$ with domination
number $\gamma(G)$, $D$ is a dominating set of $G$ with cardinality $\gamma(G)$. The proof of Theorem \ref{th3,1}
offer a method to find a unicyclic spanning subgraph $H$ with $g(H)=g_{o}(G)$ in which $D$ is also a dominating set.
For an example,
seeing Fig. 3.4, it can be checked that $\gamma(G)=3$, and $D=\{v_{1}, v_{4}, v_{6}\}$ is a dominating set. With
the method in the proof of Theorem \ref{th3,1}, we can find that $H$ is a unicyclic spanning subgraph of $G$ with
$\gamma(H)=\gamma(G)$, and find that $D$ is also a dominating set of $H$.

\begin{theorem}\label{th3,2} %------
Suppose that $v$ is a pendant neighbor in a graph $G$. There must be a dominating set of $G$ with cardinality $\gamma(G)$
containing $v$ but no any pendant vertex adjacent to $v$.
\end{theorem}

\begin{proof}
Suppose that $D$ is a dominating set of $G$ with cardinality $\gamma(G)$. If $v\in D$, we see that $D$
contains no any pendant vertex adjacent to $v$. Otherwise, by deleting the pendant vertices adjacent to $v$ from $D$,
we can get a dominating set with less cardinality than $D$, which contradicts that $|D|=\gamma(G)$.

If $v\notin D$, we assume that $u_{1}$, $u_{2}$,
$\ldots$, $u_{k}$ are all the pendant vertices adjacent to $v$. Then $u_{1}$, $u_{2}$,
$\ldots$, $u_{k}$ must be in $D$. Let $S=(D\setminus \{u_{1}$, $u_{2}$, $\ldots$, $u_{k}\})\cup \{v\}$. Then
$S$ is also a dominating set of $G$, and then
$|S|\leq |D|$. In particular, if $k\geq 2$, then $|S|< |D|$, which contradicts $|D|=\gamma(G)$.
As a result, $k= 1$ and $|S|=\gamma(G)$.
Then the result follows from the fact that $S$ is a dominating set of $G$ containing $v$.
This completes the proof. \ \ \ \ \ $\Box$
\end{proof}

In fact, by Theorem \ref{th3,2} and its proof, we have the following corollary further.

\begin{corollary}\label{cl3,3} %------
Suppose a graph $G$ contains pendant vertices. Then

$\mathrm{(i)}$ there must be a dominating set of $G$ with cardinality $\gamma(G)$ containing
all of its pendant neighbors but no any pendant vertex;

$\mathrm{(ii)}$ if $v$ is a pendant neighbor of $G$ and at least two pendant vertices are adjacent to $v$, then any
dominating set of $G$ with cardinality $\gamma(G)$ contains $v$ but no any pendant vertex adjacent to $v$.
\end{corollary}

\begin{theorem}\label{th3,4} %------
Let $3\leq s\leq n-2$ be odd, and let both $C_{s,\, l}^{\ast}$ and $C_{s,\, l+1}^{\ast}$ be of order $n$. Then
$\gamma(C_{s,\, l}^*)\leq \gamma(C_{s,\, l+1}^*).$
\end{theorem}

\begin{proof}
Suppose that the vertices of $C_{s,\, l}^{\ast}$ are indexed as in Fig. 1.1, and suppose that
$C_{s,\, l+1}^*=C_{s,\, l}^{\ast}-\sum^{n}_{s+l+2}v_{s+l}v_{i}+\sum^{n}_{s+l+2}v_{s+l+1}v_{i}$. By Theorem \ref{th3,2},
for $C_{s,\, l+1}^*$, there exists a dominating set $D$ with cardinality $\gamma(C_{s,\, l+1}^*)$
containing $v_{s+l+1}$.

{\bf Case 1} $D$ contains $v_{s+l}.$ Then $D^{'}=D\setminus \{v_{s+l+1}\}$ is a dominating set of $C_{s,\, l}^{\ast}$.
Note that $|D^{'}|\geq \gamma(C_{s,\, l}^*)$. Consequently, $\gamma(C_{s,\, l}^*)< \gamma(C_{s,\, l+1}^*)$.

{\bf Case 2} $D$ does not contain $v_{s+l}$. Let $D^{'}=(D\setminus \{v_{s+l+1}\})\cup \{v_{s+l}\}$. Then $D^{'}$ is a
dominating set of $C_{s,\, l}^{\ast}$. Note that $|D^{'}|\geq \gamma(C_{s,\, l}^*)$. Consequently,
$\gamma(C_{s,\, l}^*)\leq \gamma(C_{s,\, l+1}^*)$. Then the result follows. \ \ \ \ \ $\Box$

\end{proof}

\begin{theorem}\label{th3,5} %------
For $k\ge 2$, let both $C_{2k+1,\, l}^{\ast}$ and $C_{3,\, t}^*$ ($t=l+k-1$) be of order $n$. Then
$\gamma(C_{3,\, t}^*)\leq \gamma(C_{2k+1,\, l}^{\ast})$ (see Fig. 2.2).
\end{theorem}

\begin{proof}
Suppose the vertices of $C_{2k+1,\, l}^{\ast}$ are indexed as in Fig. 2.2, where $a=2k+1+l$, and suppose
that $\displaystyle C_{3,\, t}^*=C_{2k+1,\, l}^{\ast}-\sum^{2k+1}_{i=k+3}v_{i-1}v_{i}-v_{1}v_{2k+1}
+\sum^{2k+1}_{i=k+3}v_{a}v_{i}+v_{k+2}v_{k}$.
For graph $C_{2k+1,\, l}^{\ast}$, by Theorem \ref{th3,2}, we know that there is a dominating set $D$ containing $v_{a}$
but no any pendant vertex adjacent to $v_{a}$. We say that there is at least one of $v_{k}$,
$v_{k+1}$, $v_{k+2}$, $v_{k+3}$ is in $D$.
Otherwise, in $C_{2k+1,\, l}^{\ast}$, $v_{k+1}$, $v_{k+2}$ are not dominated by any vertex in $D$.
Suppose that at least one of $v_{k}$, $v_{k+1}$ is in $D$.
Let $\displaystyle G^{'}=C_{2k+1,\, l}^{\ast}-\sum^{2k+1}_{k+3}v_{i}$. We denote by $D_{G^{'}}=D\cap V(G^{'})$.

{\bf Case 1} At least one of $v_{1}$, $v_{2}$, $v_{2k+2}$ is in $D$. Then $D_{G^{'}}$ is also a dominating set of
$C_{3,\, t}^*$. As a result, $\gamma(C_{3,\, t}^*)\leq \gamma(C_{2k+1,\, l}^{\ast})$.

{\bf Case 2} None of $v_{1}$, $v_{2}$, $v_{2k+2}$ is in $D$. Then $v_{2k+1}$ must be in $D$. Otherwise,
in $C_{2k+1,\, l}^{\ast}$, $v_{1}$
is not dominated by any vertex in $D$. Note that $v_{2k+1}\notin D_{G^{'}}$.
Consequently, $|D|\geq |D_{G^{'}}|+1$.
Let $S=D_{G^{'}}\cup \{v_{1}\}$. Then $S$ is a dominating set of $C_{3,\, t}^*$. As a result,
$\gamma(C_{3,\, t}^*)\leq \gamma(C_{2k+1,\, l}^{\ast})$.

By Case 1 and Case 2, the result follows. \ \ \ \ \ $\Box$
\end{proof}

We say that a graph is $claw$-$free$ if it contains no induced subgraph isomorphic to $K_{1,3}$.
An $independent$ $set$ of a graph is a vertex set in which no two vertices are adjacent.
An $independent$ $dominating$ $set$ of $G$ is a vertex set that is
both dominating set and independent set of $G$. The $independent$ $domination$ $number$ of $G$, denoted by $i(G)$,
is the minimum cardinality of all independent dominating sets. In \cite{WGMH}, W. Goddarda, M. A. Henning shew that
for the path, $i(P_{n}) =\lceil\frac{n}{3}\rceil$. In \cite{ARL}, R.B. Allan, R. Laskar shew that
if $G$ is a claw-free graph, then $\gamma(G) = i(G)$. From these results, noting that a path $P_{n}$
is claw-free, we have the following Lemma \ref{le3,6}.

\begin{theorem}\label{th3,6,0} %------
Let $G$ be a unicyclic nonbipartite graph with domination number $\gamma$. Then there must be a $C_{3,\, l}^*$ with
$\gamma(C_{3,\, l}^*)=\gamma$ and with the same order as $G$.
\end{theorem}

\begin{proof}
Without loss of generality, suppose that $v_{a_{1}}$, $v_{a_{2}}$, $\ldots$, $v_{a_{k}}$ ($k\geq 2$) are
all the pendant neighbors in $G$,
and suppose that $v_{k_{1}}$, $v_{k_{2}}$, $\ldots$, $v_{k_{t}}$ are all the pendant vertices attaching
to $v_{a_{k}}$. By Corollary \ref{cl3,3}, we know that there exists a dominating set $D$ of $G$ which contains all of
its pendant neighbors but
no any pendant vertex. Let $\displaystyle G^{'}=G-\sum^{t}_{j=1}v_{a_{k}}v_{k_{j}}+\sum^{t}_{j=1}v_{a_{1}}v_{k_{j}}$. Then $D$
is also a dominating set of $G^{'}$. Therefore,
$\gamma(G^{'})\leq \gamma(G)$. But the number of the pendant neighbors of $G^{'}$ is less than that of $G$.
Proceeding like this, we can get a $\mathbb{G}$ such that $\gamma(\mathbb{G})\leq \gamma(G)$ where $\mathbb{G}$ contains only
one pendant neighbor. In fact, $\mathbb{G}\cong C_{s,\, l}^*$ for some $l$. Then the result follows from Theorems
\ref{th3,4}, \ref{th3,5}.
 This completes the proof. \ \ \ \ \ $\Box$
\end{proof}

\begin{lemma}\label{le3,6} %------
For a path $P_{n}$, we have
$\gamma(P_{n})=\lceil\frac{n}{3}\rceil$.
\end{lemma}

\begin{theorem}\label{th3,7} %------
$\gamma(C_{3,\, l}^*)=\gamma(P_{l+3})$.
\end{theorem}

\setlength{\unitlength}{0.5pt}
\begin{center}
\begin{picture}(370,127)
\put(7,101){\circle*{4}}
\put(7,34){\circle*{4}}
\qbezier(7,101)(7,68)(7,34)
\put(64,62){\circle*{4}}
\qbezier(7,34)(35,48)(64,62)
\qbezier(7,101)(35,82)(64,62)
\put(168,62){\circle*{4}}
\qbezier(64,62)(116,62)(168,62)
\put(337,61){\circle*{4}}
\put(182,62){\circle*{4}}
\put(204,62){\circle*{4}}
\put(193,62){\circle*{4}}
\put(116,62){\circle*{4}}
\put(215,62){\circle*{4}}
\put(264,62){\circle*{4}}
\qbezier(215,62)(239,62)(264,62)
\put(331,112){\circle*{4}}
\qbezier(264,62)(297,87)(331,112)
\put(335,93){\circle*{4}}
\qbezier(264,62)(299,78)(335,93)
\put(337,71){\circle*{4}}
\put(336,81){\circle*{4}}
\put(340,45){\circle*{4}}
\qbezier(264,62)(302,54)(340,45)
\put(65,67){$v_{1}$}
\put(3,107){$v_{2}$}
\put(1,19){$v_{3}$}
\put(113,68){$v_{4}$}
\put(248,48){$v_{3+l}$}
\put(333,116){$v_{3+l+1}$}
\put(339,92){$v_{3+l+2}$}
\put(345,38){$v_{n}$}
\put(130,-9){Fig. 3.5. $C_{3,\, l}^*$}
\end{picture}
\end{center}

\begin{proof}
Suppose the vertices of $C_{3,\, l}^*$ are indexed as in Fig. 3.5. As Theorem \ref{th3,2} and Corollary \ref{cl3,3},
we can get that for $C_{3,\, l}^*$, there exists a dominating set $D$ with cardinality
$\gamma(C_{3,\, l}^*)$ containing $v_{1}$ and $v_{3+l}$, but no $v_{2}$, $v_{3}$ and any pendant vertex.
Let $P=v_{3}v_{1}v_{4}v_{5}\cdots v_{3+l}v_{3+l+1}$. Note that $D$ is also a dominating set of $P$.
As a result, $\gamma(C_{3,\, l}^*)\geq\gamma(P_{l+3})$.

Conversely, by Corollary \ref{cl3,3}, for the path $P=v_{3}v_{1}v_{4}v_{5}\cdots v_{3+l}v_{3+l+1}$,
there exists a dominating set $D_{P}$ with cardinality
$\gamma(P)$ containing both $v_{1}$ and $v_{3+l}$ but no $v_{3}$, $v_{3+l+1}$. Note that $D_{P}$ is also a
dominating set of $C_{3,\, l}^*$. Consequently, $\gamma(C_{3,\, l}^*)\leq\gamma(P_{l+3})$.

From above discussion, we get that $\gamma(C_{3,\, l}^*)=\gamma(P_{l+3})$. This completes the proof. \ \ \ \ \ $\Box$
\end{proof}

By Theorem \ref{th3,1}, \ref{th3,4}-\ref{th3,6,0}, \ref{th3,7} and Lemma \ref{le3,6}, we get the following Theorem \ref{cl3,8}.

\begin{theorem}\label{cl3,8} %------
For a nonbipartite graph with both order $n$ and domination number $\gamma$, we have
$n\geq 3\gamma-1$. In particular, the equality holds for a $C_{3,\, 3\gamma-5}^*$ which has $3\gamma-1$ vertices.
\end{theorem}

\section{Minimizing the least $Q$-eigenvalue}

\begin{theorem}\label{th4,1} %------
Among all the nonbipartite unicyclic graphs with both order $n$ and domination number $\gamma$, we have

\normalfont (i) if $n=3\gamma-1$, $3\gamma$, $3\gamma+1$, then the graph with the minimal least $Q$-eigenvalue
attains uniquely at  $C_{3,\, n-4}^*$;

\normalfont (ii) if $n\geq3\gamma+2$, then the graph with the minimal least $Q$-eigenvalue attains uniquely at
$C_{3,\, 3\gamma-3}^*$.
\end{theorem}

\begin{proof}
We first claim that among all the nonbipartite unicyclic graphs with both order $n$ and domination number at most $\gamma$,
the graph with the minimal least $Q$-eigenvalue has only one pendant neighbor. Otherwise, assume that among all the nonbipartite
unicyclic graphs with both order $n$ and domination number at most $\gamma$, the graph $G$ has the minimal least $Q$-eigenvalue,
but $G$ has at least $2$ pendant neighbors. Suppose
$\mathcal {C}_{k}$ is the unique cycle in $G$,
and suppose $G=\mathcal {C}^{(T_{1}, T_{2}, \ldots, T_{t};i_{1}, i_{2}, \ldots, i_{t})}_{k}$,
where $k=g_{o}(G)$ and for $1\leq s< j\leq t$, $i_{s}\neq i_{j}$.
Suppose that $X=(x(v_1)$, $x(v_2)$, $x(v_3)$, $\ldots$, $x(v_n))^T$ is a unit eigenvector of $G$
corresponding to $\kappa(G)$. By Lemma \ref{le2,7}, we know that $\max\{|x(v_{i_{j}})|\, |\, 1\leq j\leq t\}>0$.
Suppose that $v_{p_{1}}$, $v_{p_{2}}$, $\ldots$, $v_{p_{a}}$ ($a\geq 2$) are all the pendant neighbors in $G$.
By Corollary \ref{cl3,3}, we know that there exists a dominating set $D$ of $G$ which contains all of its pendant neighbors but
no any pendant vertex. Suppose $|x(v_{p_{1}})|=\max\{|x(v_{p_{i}})|\,|\, 1\leq i\leq a\}$. Note that
$\max\{|x(v_{i_{j}})|\, |\, 1\leq j\leq t\}>0$  and $\mathcal {C}_{k}$ contains no pendant neighbor. By Lemma \ref{le2,4},
we get $|x(v_{p_{1}})|> 0$. Suppose that $v_{a_{1}}$, $v_{a_{2}}$, $\ldots$, $v_{a_{c}}$ are all the pendant vertices attaching
to $v_{p_{a}}$. Let $\displaystyle G^{'}=G-\sum^{c}_{i=1}v_{p_{a}}v_{a_{i}}+\sum^{c}_{i=1}v_{p_{1}}v_{a_{i}}$. Note that $D$
is also a dominating set of $G^{'}$. Then
$\gamma(G^{'})\leq \gamma(G)$.
By Lemma \ref{le2,3}, we get that $\kappa(G^{'})<\kappa(G)$, which contradicts the minimality
of $\kappa(G)$. As a result, our claim holds.
Then the result
follows from Lemmas \ref{le2,11}, \ref{le2,13}, Theorems
\ref{th3,4}, \ref{th3,5}, \ref{th3,7}, \ref{cl3,8} and Lemma \ref{le3,6}.  \ \ \ \ \ $\Box$
\end{proof}

Moreover, by Lemma \ref{le2,0}, Theorems \ref{th3,1} and \ref{th4,1}, we get the following result.

\begin{theorem}\label{th2,18} %------
Among all the nonbipartite graphs with both order $n$ and domination number $\gamma$, we have

\normalfont (i) if $n=3\gamma-1$, $3\gamma$, $3\gamma+1$, then the graph with the minimal least $Q$-eigenvalue attains uniquely
at  $C_{3,\, n-4}^*$;

\normalfont (ii) if $n\geq3\gamma+2$, then the graph with the minimal least $Q$-eigenvalue attains uniquely at
$C_{3,\, 3\gamma-3}^*$.
\end{theorem}

\small {

}

\end{document}